\documentclass[10pt, reqno]{amsart}
\usepackage[utf8]{inputenc}

\usepackage{graphicx}
\usepackage{comment}
\usepackage{fourier}
\usepackage[T1]{fontenc}
\usepackage[margin=0.75in]{geometry}
\usepackage{parskip}
\usepackage{amsthm}
\usepackage{amsmath}
\usepackage{amssymb}
\usepackage{mathtools}
\usepackage{apptools}
\usepackage{setspace}
\usepackage{url}
\usepackage{tabularx}
\usepackage[
backend=biber,
style=numeric,
sorting=nyt,
maxnames=10,
giveninits=true
]{biblatex}
\newcommand {\R}{\mathbb{R}}

\newcommand {\Z}{\mathbb{Z}}
\newcommand {\N}{\mathbb{N}}

\newcommand {\C}{\mathbb{C}}

\newcommand{\la}{\lambda}

\newcommand{\sumast}{\mathop{\sum\nolimits^{\mathrlap{\ast}}}}
\newcommand{\sumdag}{\mathop{\sum\nolimits^{\mathrlap{\dagger}}}}
\newcommand{\Mod}[1]{\ (\mathrm{mod}\ #1)}
\newtheorem{thm}{Theorem}[section]

\newtheorem{lemma}[thm]{Lemma}

\newtheorem{cor}{Corollary}[thm]
\newtheorem*{remark}{Remark}
\newtheoremstyle{named}{}{}{\itshape}{}{\bfseries}{.}{.5em}{\thmnote{#3}}
\theoremstyle{named}

\AtAppendix{\counterwithin{lemma}{section}}
\numberwithin{equation}{section}

\usepackage{multicol}
\usepackage{xcolor}
\usepackage{hyperref}
\hypersetup{
    colorlinks,
    linkcolor={red!50!black},
    citecolor={blue!50!black},
    urlcolor={blue!80!black}
}

\title{Shifted convolution sums for $GL(3)\times GL(2)$ averaged over weighted sets}
\author{Wing Hong Leung}

\begin{document}

\maketitle

\begin{abstract}
    Let $A(1,m)$ be the Fourier coefficients of a $SL(3,\Z)$ Hecke-Maass cusp form $F$ and $\la(m)$ be those of a $SL(2,\Z)$ Hecke holomorphic or Hecke-Maass cusp form $g$. Let $\mathcal{H}\subset\llbracket -X^{1-\varepsilon},X^{1+\varepsilon}\rrbracket$ and $\{a(h)\}_{h\in\mathcal{H}}\subset\C$ be a sequence. We show that if $\mathcal{H}\subset \ell+\llbracket 0,X^{1/2+\varepsilon}\rrbracket $ for some $\ell\geq0$, \begin{align*}
        D_{a,\mathcal{H}}(X):=\frac{1}{|\mathcal{H}|}\sum_{h\in\mathcal{H}}a(h)\sum_{m=1}^\infty A(1,m)\la(rm+h)V\left(\frac{m}{X}\right)\ll_{F,g,\varepsilon} \frac{X^{1+\varepsilon}}{|\mathcal{H}|}\|a\|_2
    \end{align*}
    for any $\varepsilon>0$, and a similar bound holds when $\mathcal{H}\supset \ell+\llbracket 0,X^{1/2+\varepsilon}\rrbracket$. This improves Sun's bound and generalizes it to an average with arbitrary weights $a(h)$. Moreover, we demonstrate how one can recover the factorizable moduli structure given by Jutila's variant of the circle method via studying a shifted sum with weighted average. This allows us to recover Munshi's bound on the shifted sum with a fixed shift without using Jutila's circle method.
\end{abstract}

\section{Introduction}

The shifted convolution sum problem concerns sums of the form \begin{align*}
    \sum_{m=1}^\infty a(m)b(m+\ell)V\left(\frac{m}{X}\right),
\end{align*}
for some smooth function $V$ compactly supported on $\R^+$. Such a sum is related to many different problems by specializing the sequences $a$ and $b$, and many cases have been extensively studied. Obtaining a non-trivial bound often yields deep applications. For example, if $a$ and $b$ are specialized to be combinations of von Mangoldt function, M\"{o}bius function and/or the divisor function, these shifted sums is related to classical problems including the twin prime conjecture, Chowla's conjecture, moments of zeta function and many others. When $a$ and $b$ are specialized to be coefficients of $\mathrm{GL}(2)$ forms, these shifted sums are related to subconvexity and Quantum Unique Ergodicity, see for example \cite{blomer2004shifted}, \cite{duke1993bounds}, \cite{harcos2003additive}, \cite{Ho}, \cite{holowinsky2012level}, \cite{KMV}, \cite{leung2022reformulation}, \cite{LS}, \cite{michel2004subconvexity}, \cite{sarnak2001estimates}, \cite{selberg1965estimation}.

We now consider a higher rank analogue, when $a$ is a $\mathrm{GL}(3)$ coefficient and $b$ is a $\mathrm{GL}(2)$ coefficient. Pitt \cite{pitt1995shifted} studied the shifted sum between the ternary divisor function and a $\mathrm{GL}(2)$ coefficient. Let $A(1,m)$ be the Fourier coefficients of a $SL(3,\Z)$ Hecke-Maass cusp form $F$ and $\la(m)$ be those of a $SL(2,\Z)$ Hecke holomorphic or Hecke-Maass cusp form $g$. For $a$ coming from a $SL(3,\Z)$ cusp form, the first breakthrough is done by Munshi in \cite{munshi2013shifted}, who obtained (after a minor correction in his Lemma 11) \begin{align*}
    D_h(X)=\sum_{m=1}^\infty A(1,m)\la(m+h)V\left(\frac{m}{X}\right)\ll_{F,g,\varepsilon} X^{1-1/26+\varepsilon}.
\end{align*}
He used Jutila's variant of the circle method in order to have factorizable moduli, which provides him a way to better balance the mass between the diagonal and offdiagonal contribution and obtained the first non-trivial bound. He also showed a similar result for the ternary divisor function case in \cite{munshi2013ashifted}. This result was later improved by Xi in \cite{xi2018shifted}, who showed that \begin{align*}
    D_h(X)\ll_{F,g,\varepsilon} X^{1-1/22+\varepsilon}.
\end{align*}
He used a similar starting point as Munshi's idea, exploiting the factorizable moduli structure given by Jutila's circle method. Additionally, he used an exponent pair argument that works on moduli that are smooth numbers to obtain the extra saving.

In the mean time, inspired by a comment made by Munshi in \cite{munshi2013shifted}, Sun studied the shifted sum with an average over the shifts in \cite{sun2018averages}. For any $A>0$ and integer $r\geq1$, she showed that \begin{align}\label{SunAverage}
    \frac{1}{H}\sum_{h=1}^\infty W\left(\frac{h}{H}\right)\sum_{m=1}^\infty A(1,m)\la(rm+h)V\left(\frac{m}{X}\right)\ll_{F,g,\varepsilon} \begin{cases} X^{-A} & \text{ if } (rX)^{1/2+\varepsilon}\leq H\leq X\\
    X^{1-\delta+\varepsilon} & \text{ if } r^{5/2}X^{1/4+7\delta/2}\leq H\leq (rX)^{1/2+\varepsilon}.\end{cases}
\end{align}

\subsection{Main results}

Let $r>0$ be a fixed positive integer. Let $\varepsilon>0$, $V\in C_c^\infty([1,2])$ be a fixed smooth function. Let $\mathcal{H}\subset\llbracket -X^{1-\varepsilon},X^{1+\varepsilon}\rrbracket$ be a set and $\{a(h)\}_{h\in\mathcal{H}}$ be a sequence of complex numbers. Here $\llbracket a,b\rrbracket$ denotes the set $\{x\in\Z: a\leq x< b\}$. Write $\|a\|_2^2=\displaystyle\sum_{h\in\mathcal{H}}|a(h)|^2$. In this paper, we study the $GL(3)\times GL(2)$ shifted convolution sum averaged over weighted sets, \begin{align}
    D_{a,\mathcal{H}}(X):=\frac{1}{|\mathcal{H}|}\sum_{h\in\mathcal{H}}a(h)\sum_{m=1}^\infty A(1,m)\la(rm+h)V\left(\frac{m}{X}\right).
\end{align}

\begin{remark}
    Fixing $r$ here is purely due to aesthetic reasons. One can easily generalize the arguments here to the case where $r>0$ is varying, but the statements and proofs will be messier.
\end{remark}

The main goal of this paper is twofold. The \textbf{first goal} is to prove the following two main theorems, which improves Sun's bound in \cite{sun2018averages} and generalize it to weighted shifts averaged over an arbitrary set.

\begin{thm}\label{mainthm}
    Let $\varepsilon>0$ and let $\mathcal{H}\subset\llbracket -X^{1-\varepsilon},X^{1+\varepsilon}\rrbracket$. We have \begin{align*}
        D_{a,\mathcal{H}}(X)\ll_{f,g,\varepsilon} \frac{ X^{3/4+\varepsilon}}{|\mathcal{H}|}\|a\|_2\left(\sum_{b\ll X^{1/2+\varepsilon}}\sup_{h\in\mathcal{H}}\left|\{h'\in\mathcal{H}: h'\equiv h\Mod{b})\right|\right)^{1/2}.
    \end{align*}
    In particular, if $\mathcal{H}\subset \ell+\llbracket 0,X^{1/2+\varepsilon}\rrbracket $ for some $\ell\geq0$, then we have \begin{align*}
        D_{a,\mathcal{H}}(X)\ll_{F,g,\varepsilon} \frac{X^{1+\varepsilon}}{|\mathcal{H}|}\|a\|_2;
    \end{align*}
    and if $\mathcal{H}\supset \ell+\llbracket 0,H\rrbracket$ for some $\ell\geq0$ and $H\gg X^{1/2+\varepsilon}$, then we have \begin{align*}
        D_{a,\mathcal{H}}(X)\ll_{F,g,\varepsilon} X^{3/4+\varepsilon}\frac{\|a\|_2}{\sqrt{|\mathcal{H}|}}.
    \end{align*}
\end{thm}

Compared to the current best bound $X^{1-1/22+\varepsilon}$ for a fixed shift (meaning $h$ is a fixed integer instead of being averaged over a set $\mathcal{H}$) obtained by Xi in \cite{xi2018shifted}, this shows that having an average of shifts is beneficial when $|\mathcal{H}|\gg X^{1/11}$. Taking the special case $\mathcal{H}=\llbracket 1,H\rrbracket$ and $a(h)=W(h/H)$ for some function $W\in C_c^\infty(\R^+)$, Theorem \ref{mainthm} then implies that the average shifted sum in \eqref{SunAverage} is bounded by $X^{1+\varepsilon}/\sqrt{H}$ for $H\ll X^{1/2+\varepsilon}$, thus improving the bound \eqref{SunAverage} obtained by Sun in \cite{sun2018averages}\footnote{It is easy to obtain the arbitrary saving part when $H\gg X^{1/2+\varepsilon}$ and it is also covered as a special case of Theorem \ref{mainthm2} shown below.}.

If $\mathcal{H}$ and $\{a(h)\}$ are `factorizable' in the following manner: Let $D,Q_1,Q_2>0$ such that $DQ_1Q_2\ll X^{1-\varepsilon}$ and let $\mathcal{Q}_1\subset\llbracket 1,Q_1\rrbracket$. Let $\{a'(q_1)\}_{q_1\in\mathcal{Q}_1}$ be a sequence of complex numbers. Let $V_1,V_2\in C_c^\infty(\R^+)$ be fixed functions. Let $0\leq \ell\ll X^{1+\varepsilon}$ and suppose $a(h)$ is of the form \begin{align}
    a_\pm(h)=\mathop{\sum\sum\sum}_{\substack{q_1\in\mathcal{Q}_1\\\pm dq_1q_2=h-\ell}}V_1\left(\frac{d}{D}\right)V_2\left(\frac{q_2}{Q_2}\right)a'(q_1)\tag{$\ast$}
\end{align}
and \begin{align*}
    \mathcal{H}\supset \{h\in\Z: a_\pm(h)\neq0\},
\end{align*}
then we have the following theorem.

\begin{thm}\label{mainthm2}
    Let $\varepsilon>0$. Suppose $\mathcal{H}$ and $\{a_\pm(h)\}$ satisfy $(*)$. If $D+Q_1\gg X^{1/2+\varepsilon}$, we have \begin{align*}
        D_{a_\pm,\mathcal{H}}(X)\ll_{F,g,\varepsilon,A}X^{-A}
    \end{align*}
    for any $A>0$. On the other hand, if $D+Q_1\ll X^{1/2+\varepsilon}$ and $\mathcal{Q}_1\subset \{1 \text{ or Primes in }[Q_1,2Q_1]\}$ with $Q_1Q_2\gg X^{1/2+\varepsilon}$, we have \begin{align*}
        D_{a_\pm,\mathcal{H}}(X)\ll_{F,g,\varepsilon}\frac{X^{1+\varepsilon}}{\sqrt{|\mathcal{H}|}}\left(1+\frac{\sqrt{X}Q_1}{DQ_2}\right)^{1/2}\|a\|_\infty,
    \end{align*}
    where $\displaystyle\|a\|_\infty=\sup_{h\in\mathcal{H}}\{|a_\pm(h)|\}$.
\end{thm}

The special case $(*)$ is of particular interest to us, and it is related to our \textbf{second goal} of this paper: To recover the factorizable moduli structure of Jutila's circle method, which we now explain and exemplify with the fixed shift problem.

Let $0\leq \ell\ll X^{1+\varepsilon}$, write $Y=X+\ell$ and consider \begin{align*}
    D_\ell(X)=\sum_{m=1}^\infty A(1,m)\la(m+\ell)V\left(\frac{m}{X}\right)
\end{align*}
as before. Let $\varphi \in C_c^\infty([1/2,5/2])$ be such that $\varphi (x)=1$ for $1\leq x\leq 2$. Applying a divisor switching or hyperbola trick similar to the initial steps on how the reformulation of the Delta method is proved (see Lemma \ref{DeltaMethod} and appendix \ref{sect.ProofDeltaMethod}), we obtain the following simple Lemma.

\begin{lemma}\label{mainlemma}
    Let $q$ be a positive integer and $D>0$ be a parameter. Let $H$ be any function such that $H(0)=1$. Then we have \begin{align*}
        \delta(n=0)=&\delta(n\equiv 0\Mod{q})H\left(\frac{n}{Dq}\right)-\sum_{0\neq d\in\Z}\delta(n=dq)H\left(\frac{d}{D}\right).
    \end{align*}
    In particular, for any set of positive integers $\mathcal{Q}$ with any sequence $\{b(q)\}_{q\in\mathcal{Q}}$, we have \begin{align*}
        D_\ell(X)= M.T. - A.S.^+-A.S.^-,
    \end{align*}
    where \begin{align*}
        M.T.:=\frac{1}{\mathfrak{Q}}\sum_{q\in\mathcal{Q}}b(q)\sum_{m=1}^\infty A(1,m)V\left(\frac{m}{X}\right)\sum_{n=1}^\infty\la(n)\varphi \left(\frac{n}{Y}\right)\delta(m+\ell\equiv n\Mod{q})H\left(\frac{m+\ell-n}{Dq}\right)
    \end{align*}
    and \begin{align*}
        A.S.^\pm:=\frac{1}{\mathfrak{Q}}\sum_{q\in\mathcal{Q}}b(q)\sum_{d=1}^\infty\sum_{m=1}^\infty A(1,m)V\left(\frac{m}{X}\right)\la(m+\ell\pm dq)\varphi \left(\frac{m+\ell\pm dq}{Y}\right)H\left(\pm\frac{d}{D}\right),
    \end{align*}
    with $\mathfrak{Q}=\sum_{q\in\mathcal{Q}}b(q)$.
\end{lemma}
\begin{proof}
    Using $H(0)=1$, we have
    \begin{align*}
        \sum_{0\neq d\in\Z}\delta(n=dq)H\left(\frac{d}{D}\right)=&\sum_{0\neq d\in\Z}\delta(n=dq)H\left(\frac{n}{Dq}\right)\\
        =&\sum_{d\in\Z}\delta(n=dq)H\left(\frac{n}{Dq}\right)-\delta(n=0)\\
        =&\delta(n\equiv 0\Mod{q})H\left(\frac{n}{Dq}\right)-\delta(n=0).
    \end{align*}
    To get the second statement, note that $\varphi((m+\ell)/Y)=1$ when $X\leq m\leq 2X$ and apply the above expression of the delta symbol to \begin{align*}
        D_\ell(X)=&\sum_{m=1}^\infty A(1,m)\la(m+\ell)V\left(\frac{m}{X}\right)\varphi\left(\frac{m+\ell}{Y}\right)\\
        =&\sum_{m=1}^\infty A(1,m)V\left(\frac{m}{X}\right)\sum_{n=1}^\infty \la(n)\varphi\left(\frac{n}{Y}\right)\delta(m+\ell-n=0).
    \end{align*}
    Averaging over $q\in \mathcal{Q}$ with weights $b(q)$ yields the second statement.
\end{proof}

With \begin{align*}
    \delta(m+\ell\equiv n\Mod{q})=\frac{1}{q}\sum_{\alpha\Mod{q}}e\left(\frac{\alpha(m+\ell-n)}{q}\right)=\frac{1}{q}\sum_{q_0q'=q}\sumast_{\alpha\Mod{q'}}e\left(\frac{\alpha(m+\ell-n)}{q'}\right),
\end{align*}
$M.T.$ is almost the main term when one applies Jutila's circle method, with the only difference being the presence of the $q_0$-sum. In application, such a difference is easy to handle by choosing $\mathcal{Q}$ carefully. On the other hand, $A.S.^\pm$ is precisely a shifted sum with the shifts summing over the set $\mathcal{H}=\{\ell\pm dq:H(d)\neq0, q\in\mathcal{Q}\}$ with weights $a(h)=\sumdag_{dq=h}H\left(\pm\frac{d}{D}\right)$, with the sum going over $\{d:H(d)\neq0\}$ and $q\in\mathcal{Q}$. Applying a dyadic subdivision on the $d$-sum followed by Theorem \ref{mainthm2}, we recover Munshi's result on the $GL(3)\times GL(2)$ shifted convolution sum for a fixed shift by analysing the average version of the shifted sum, instead of invoking Jutila's circle method.

\begin{thm}(\cite[Thm 1]{munshi2013shifted})
    For $0<h\ll X^{1+\varepsilon}$, \begin{align*}
        \mathcal{D}_h(X)\ll X^{1-1/26+\varepsilon}.
    \end{align*}
\end{thm}

Looking at the shifted convolution sum problem $\mathcal{D}_h(X)$ without an average over the shift $h$, we have demonstrated how one can obtain an arbitrary moduli structure allowed in Jutila's circle method by Lemma \ref{mainlemma}: One can choose the set of moduli $\mathcal{Q}$ freely, $M.T.$ resembles the main term in Jutila's circle method, and one studies the average shifted sum $A.S.^\pm$ instead of the implicit error term given in Jutila's circle method. This framework works in general, as depicted by the simplicity of Lemma \ref{mainlemma}. Surprisingly in the case of $\mathrm{GL}(3)\times\mathrm{GL}(2)$ shifted sum, the bound coming from the average shifted sum we obtain in Theorem \ref{mainthm2} matches with the error term treatment in Jutila's circle method, and hence giving the exact same bound Munshi obtained in \cite{munshi2013shifted} (after fixing the error in his Lemma 11).\footnote{The moduli structure we have here is essentially of the form $[Q_1,2Q_1]\cdot \{$Primes in $[Q_2,2Q_2]\}$ for some $Q_1,Q_2>0$. This is technically different from Munshi's choice of the form $\{$Primes in $[Q_1,2Q_1]\}\cdot \{$Primes in $[Q_2,2Q_2]\}$. However, the difference is technical but not essential.} (See Section \ref{sect.RecoveringMunshi} for details.) However, the method to bound $A.S.^\pm$ is quite different from how the error term in Jutila's circle method is treated, which has the following two implications. \begin{itemize}
    \item One can improve the resulting bound by improving the bound on the average shifted sums.
    \item It may lead to different estimates if one applies this framework in other problems when Jutila's circle method is helpful.
\end{itemize}

On the other hand, recovering a special case of Jutila's circle method using our Delta method (see Section \ref{sect.Delta}) is interesting on its own. The author showed in \cite[Ch 3]{leung2022reformulation} that appropriate choices of parameters in our Delta method recovers most Delta methods\footnote{They include the trivial Delta method \cite{aggarwal2020burgess}, Kloosterman's refinement of the circle method \cite{kloostermanrefinement}, DFI Delta method \cite{DFIdelta} and some special cases of the $\mathrm{GL}(2)$ Delta method (using the Petersson trace formula as a Delta symbol).} together with a simple case of Jutila's circle method\footnote{The simple case refers to choosing the set of moduli as $\mathcal{Q}_1\cdot\llbracket 1,Q_2\rrbracket$ with $\mathcal{Q}_1\subset \llbracket Q_1,2Q_1\rrbracket$ with $Q_1Q_2^2$ larger than the size of the equation $X$. This simple case is also covered by the DFI Delta method with an appropriate conductor lowering trick.}. It is then natural to ask if we can recover other special cases of Jutila's circle method. Since our Delta method is used in proving Theorem \ref{mainthm} and \ref{mainthm2}, the above discussions provide a partial affirmative answer.

Unfortunately, we have not been able to recover Xi's $X^{1-1/22+\varepsilon}$ bound for the fixed shift problem in \cite{xi2018shifted}. This is due to the fact that his method seems to crucially rely on the choice of moduli being square-free smooth numbers. In such a case, Theorem \ref{mainthm} and the methods used in proving Theorem \ref{mainthm2} are not sufficient to yield a good enough bound for these kind of shifts.

\subsection{Two tricks}

To prove Theorem \ref{mainthm} and \ref{mainthm2}, we used the Delta method (Lemma \ref{DeltaCor}) to separate the oscillations. In the process, we used two simple tricks to obtain the main results. The first one is crucial, and the second one is used to shorten the proof. Moreover, these tricks may be of independent interest to the readers.

\begin{enumerate}
    \item At some Cauchy Schwarz inequality step, we apply a lengthening trick to the outer sum, i.e. \begin{align*}
        \sum_{m\leq M}\left|f(m)\right|^2\ll \sum_{m\leq ML}\left|f(m)\right|^2
    \end{align*}
    for any $L\geq1$. This lengthening trick can be seen as a simpler analogue of Xiannan Li's trick on his recent paper \cite{li2022moments}. By adding more summation terms outside an absolute value square, we boost the diagonal contribution while lowering the offdiagonal contribution. See Section \ref{sect.CS1} and \ref{sect.CS2} for details.
    \item To prove Theorem \ref{mainthm2}, we add a small oscillation right before we replace the Delta symbol by our Delta method, i.e. if $b>0$, we have \begin{align*}
        \delta(a=b)=\delta(a=b)(a^2/b^2)^{it}
    \end{align*}
    for any $t\in\R$. This allows us to remove certain zero frequency terms coming from Poisson summations and shortens our proof. See Lemma \ref{mainlemma2} and its proof in Section \ref{sect.Proofmainlemma2} for details.
\end{enumerate}

\noindent \textbf{Notation.} Throughout the paper, $\varepsilon>0$ is a very small number that may vary depending on context. We denote $e(x)=e^{2\pi ix}$. We write $\llbracket a,b\rrbracket$ as the set $\{x\in\Z: a\leq x<b\}$. We denote $A\cdot B=\{ab:a\in A, b\in B\}$ for any sets $A,B\subset \Z$. We say a function $f$ is $Z$-inert if it satisfies the bound $x^{-j}\frac{d^j}{dx^j}f(x)\ll_j Z^j$ for any $j\geq0$.

\section{Preliminaries}

\subsection{\texorpdfstring{$SL(3,\Z)$}{SL(3,Z)} Maass forms}

Let $F$ be a Maass form of type $(\nu_1,\nu_2)$ for $\mathrm{SL}_3(\Z)$, which is an eigenfunction for all the Hecke operators. Let the Fourier coefficients be $A(n_1,n_2)$, normalized so that $A(1,1)=1$. The Langlands parameter $(\alpha_1,\alpha_2, \alpha_3)$ associated with $F$ are $\alpha_1=-\nu_1-2\nu_2+1,\ \alpha_2= -\nu_1+\nu_2,\ \alpha_3= 2\nu_1+\nu_2-1$. We refer the reader to Goldfeld's book \cite{goldfeldbook} for more details.

The Fourier coefficients satisfy the following bound.

\begin{lemma}\label{lem.GL3RP}
    We have
    \begin{align*}
        \sum_{m^2n\ll M}|A(m,n)|^2, \sum_{m^2n\ll M}|A(n,m)|^2\ll M^{1+\varepsilon}.
    \end{align*}
\end{lemma}

Together with the Hecke relations \cite[Thm 6.4.11]{goldfeldbook} and M\"{o}bius inversion, we have the following Corollary.

\begin{cor}\label{GL3RPavg}
    We have
    \begin{align*}
        \sum_{n_1\leq N_1}\sum_{n_2\leq N_2}|A(n_1,n_2)|^2\ll (N_1N_2)^{1+\varepsilon}.
    \end{align*}
\end{cor}
\begin{proof}
    Applying the Hecke multiplicativity, we have 
    \begin{align*}
        \sum_{n_1\leq N_1}\sum_{n_2\leq N_2}|A(n_1,n_2)|^2=&\sum_{n_1\leq N_1}\sum_{n_2\leq N_2}\left|\sum_{d|(n_1,n_2)}\mu(d)A(n_1/d,1)A(1,n_2/d)\right|^2.
    \end{align*}
    Applying Cauchy-Schwarz inequality to take out the $d$-sum, the above is bounded by \begin{align*}
        \ll& X^\varepsilon\sum_{n_1\leq N_1}\sum_{n_2\leq N_2}\sum_{d|(n_1,n_2)}|A(n_1/d,1)A(1,n_2/d)|^2\\
        =&X^\varepsilon\sum_d\sum_{n_1\leq N_1/d}\sum_{n_2\leq N_2/d}|A(n_1,1)A(1,n_2)|^2.
    \end{align*}
    Then Lemma \ref{lem.GL3RP} yields the desired bound.
\end{proof}

We will also need the following Voronoi summation formula. Let $g$ be a compactly supported smooth function on  $\R^+$, and let $\tilde{g}(s)=\int_0^\infty g(x)x^{s-1}\ dx$ be its Mellin transform. For $\sigma>-1+\max\{-Re(\alpha_1), -Re(\alpha_2), -Re(\alpha_3)\}$ and $a=0, 1$, define
\begin{equation*}
\gamma_a(s) = \frac{\pi^{-3s-3/2}}{2}\prod_{i=1}^3\frac{\Gamma(\frac{1+s+\alpha_i+a}{2})}{\Gamma(\frac{-s-\alpha_i+a}{2})}.
\end{equation*}
Further set $\gamma_\pm(s) = \gamma_0(s)\mp i\gamma_1(s)$ and let
\begin{equation*}
G_\pm(y) = \frac{1}{2\pi i}\int_{(\sigma)} y^{-s}\gamma_\pm(s)\tilde{g}(-s)\ ds. 
\end{equation*}
Then we have the following lemma. (See \cite{blomer2012subconvexity}, \cite{Li1}, \cite{Miller-Schmid}).

\begin{lemma}\label{lem.GL3Voronoi}
Let $g$ be a compactly supported smooth function on $(0, \infty)$. We have
\begin{equation*}
\sum_{n=1}^\infty \lambda(1,n) e\left(\frac{an}{c}\right)g\left(\frac{n}{X}\right) = c\sum_\pm \sum_{n_0|c}\sum_{n=1}^\infty \frac{\lambda(n,n_0)}{nn_0} S(\overline{a}, \pm n; c/n_0) G_\pm \left(\frac{n_0^2nX}{c^3}\right),
\end{equation*}
where $(a,c)=1$ and $a\overline{a}\equiv 1\bmod c$.
\end{lemma}

Suppose $g$ is supported on a fixed compact set in $\R^+$ and it is $X^\varepsilon$-inert, i.e. $x^{-j}\frac{d^j}{dx^j}g(x)\ll_j X^{j\varepsilon}$ for any $j\geq0$. Then a standard integral analysis on $G_\pm(n_0^2nX/c^3)$ implies that we get arbitrary savings unless $n_0^2n\ll c^3X^{-1+\varepsilon}$. Moreover, by \cite[Lemma 7]{blomer2012subconvexity}, we have \begin{align}\label{GL3BesselDerProp}
    x^{-j}\frac{d^j}{dx^j}G_\pm(x)\ll_j (X^\varepsilon+x^{1/3})^jx^{2/3}\|g\|_\infty.
\end{align}

\subsection{\texorpdfstring{$SL(2,\Z)$}{SL(2,Z)} cusp forms}

Let $g$ be a primitive holomorphic Hecke eigen cusp form or a primitive $\,\mathrm{GL}(2)$ Hecke-Maass cusp form of full level. Let $\la_g(n)$ be its Fourier coefficients. Then $g$ has the Fourier expansion \begin{align*}
    g(z)=\sum_{n\geq1}\la_g(n)n^{\frac{k-1}{2}}e(nz)
\end{align*}
if $g$ is holomorphic of weight $k$; and \begin{align*}
    g(z)=\sqrt{y}\sum_{n\neq0}\la_g(n)K_{it}(2\pi|n|y)e(nx)
\end{align*}
if $g$ is a Maass form with Laplace eigenvalue $\frac{1}{4}+t^2\geq0$.

These Fourier coefficients satisfy the following bound.

\begin{lemma}\label{lem.GL2RPavg}
    We have
    \begin{align*}
        \sum_{n\ll N}|\la_g(n)|^2\ll N^{1+\varepsilon}.
    \end{align*}
\end{lemma}

Moreover, we have the following Voronoi summation formula.

\begin{lemma}\cite[Thm A.4]{KMV}\label{lem.GL2Voronoi}
    Let $a,c\neq0$ be integers such that $(a,c)=1$ and let $F\in C_c^\infty(\R^+)$. Then we have \begin{align*}
        \sum_{n\geq1}\la_g(n)e\left(\frac{an}{c}\right)F\left(\frac{n}{X}\right)=\frac{X}{c}\sum_{\pm}\sum_{n\geq1}\la_g(\pm n)e\left(\mp\frac{\overline{a}n}{c}\right)\int_0^\infty F(y)J_{\pm,g}\left(\frac{4\pi\sqrt{nXy}}{c}\right)dy,
    \end{align*}
    where \begin{align*}
        J_{+,g}(x)=2\pi i^k J_{k-1}\left(x\right)
    \end{align*}
    and $J_{-,g}=0$ if $g$ is holomorphic of weight $k$, and \begin{align*}
        J_{+,g}(x)=-\frac{\pi}{\sin(\pi ir)}\left(J_{2ir}(x)-J_{-2ir}(x)\right)
    \end{align*}
    and \begin{align*}
        J_{-,g}(x)=\epsilon_g4\cosh(\pi r)K_{2it}(x)
    \end{align*}
    if $g$ is Hecke-Maass with Laplace eigenvalue $\frac{1}{4}+t^2$, and $\epsilon_g=\begin{cases}1 & \text{ if } g \text{ is even}\\ -1 & \text{ if } g \text{ is odd.}\end{cases}$
\end{lemma}

If $F$ is a $Z$-inert function,
by the derivative properties of the Bessel functions, repeated integration by parts on the $y$-integral gives us arbitrary savings unless $n\ll c^2(1+Z^2)X^{-1+\varepsilon}$.

\subsection{Integral analysis}

We need the following lemma for a simple integral analysis.

\begin{lemma}\cite[Lemma 3.1]{inert}\label{lem.statphase}
    Suppose that $w$ is a $Z$-inert function with a fixed compact support on $\R^+$. Let $\phi$ be a real-valued, smooth function. Let \begin{align*}
        I=\int_{-\infty}^\infty w(t)e^{i\phi(t)}d t.
    \end{align*}
    If $|\phi'(t)|\gg ZX^\varepsilon$ for all $t$ in the support of $w$, then $I\ll_A X^{-A}$ for any $A>0$.
\end{lemma}

\section{Reformulation of the Delta method}\label{sect.Delta}

To prove Theorem \ref{mainthm} and Theorem \ref{mainthm2}, we apply our reformulation of the (DFI) Delta method. The following version can be viewed as a simplified version of the DFI delta method (see \cite{DFIdelta}, \cite{heathbrown}), and it is introduced by the author in \cite{leung2021hybrid} and \cite[Ch 3]{leung2022reformulation}. For a more general statement with its proof, see Lemma \ref{DeltaMethod} in Appendix A.

\begin{lemma}\label{DeltaCor}
    Let $\varepsilon>0$, $n,q$ be integers such that $q>0$, $|n|\ll N$ and let $C>N^\varepsilon$ be a parameter. Let $U\in C_c^\infty(\R), W\in C_c^\infty([-2,-1]\bigcup [1,2])$ be a fixed non-negative even function such that $U(x)=1$ for $-2\leq x\leq 2$. Then we have \begin{align*}
        \delta(n=0)=\frac{1}{\mathcal{C}}\sum_{c\geq1}\frac{1}{cq}\sum_{\alpha\Mod{cq}}e\left(\frac{\alpha n}{cq}\right)h\left(\frac{c}{C},\frac{n}{cCq}\right),
    \end{align*}
    with $\mathcal{C}=\displaystyle\sum_{c\geq1}W\left(\frac{c}{C}\right)\asymp C$ and \begin{align*}
        h\left(x,y\right)=W\left(x\right)U\left(x\right)U\left(y\right)-W(y)U(x)U(y).
    \end{align*}
    In particular, $h$ is a fixed smooth function satisfying $h(x,y)\ll \delta(|x|,|y|\ll 1).$
\end{lemma}

Compared with the form written in \cite{heathbrown}, Lemma \ref{DeltaCor} is essentially the DFI Delta method with a simpler weight function $h$ that restricts $|n|\ll cCq$. This particular feature enables an easier integral analysis when dual summations are applied. See \cite{leung2021hybrid} for such an application and \cite[Ch 3]{leung2022reformulation} for further discussions.

\section{Proof of Theorem \ref{mainthm}}\label{sect.ProofMainThm}

Let $Y=rX+\max\{|h|:h\in\mathcal{H}\}$, then $X\leq Y\ll X^{1+\varepsilon}$. Let $\varphi \in C_c^\infty([1/2,5/2])$ be such that $\varphi (x)=1$ for $1\leq x\leq 2$. Then we have \begin{align*}
    D_{a,\mathcal{H}}(X)=&\frac{1}{|\mathcal{H}|}\sum_{h\in\mathcal{H}}a(h)\sum_{m=1}^\infty A(1,m)\la(rm+h)V\left(\frac{m}{X}\right)\\
    =&\frac{1}{|\mathcal{H}|}\sum_{h\in\mathcal{H}}a(h)\sum_m A(1,m)V\left(\frac{m}{X}\right)\sum_n\la(n)\varphi \left(\frac{n}{Y}\right)\delta(n=rm+h).
\end{align*}

\subsection{Applying our Delta method}

Let $C>X^{1/2+\varepsilon}$ be a parameter. Applying Lemma \ref{DeltaCor} with $q=1$ and \begin{align*}
    \sum_{\alpha\Mod{c}}e\left(\frac{\alpha n}{c}\right)=\sum_{b_0b=c}\sumast_{\alpha\Mod{b}}e\left(\frac{\alpha n}{b}\right)
\end{align*}
we have\begin{align}
    D_{a,\mathcal{H}}(X)=&\frac{1}{\mathcal{C}|\mathcal{H}|}\mathop{\sum\sum}_{b_0,b\geq1}\frac{1}{b_0b}\sum_{h\in\mathcal{H}}a(h)   \sum_m A(1,m)V\left(\frac{m}{X}\right)\sum_n\la(n)\varphi \left(\frac{n}{Y}\right)\nonumber\\
    &\times \sumast_{\alpha\Mod{b}}e\left(\frac{\alpha(rm+h-n)}{b}\right)  h\left(\frac{b_0b}{C},\frac{rm+h-n}{b_0bC}\right),
\end{align}
for some fixed smooth function $h$ satisfying $h(x,y)\ll \delta(|x|,|y|\ll 1).$

\subsection{Voronoi summations}

Applying Voronoi summation to the $n$-sum with Lemma \ref{lem.GL2Voronoi} yields \begin{align*}
    D_{a,\mathcal{H}}(X)=&\frac{Y}{\mathcal{C}|\mathcal{H}|}\sum_{\eta_2=\pm1}\mathop{\sum\sum}_{b_0,b\geq1}\frac{1}{b_0b^2}\sum_{h\in\mathcal{H}}a(h)   \sum_m A(1,m)V\left(\frac{m}{X}\right)\sum_n\la(\eta_2 n)\nonumber\\
    &\times \sumast_{\alpha\Mod{b}}e\left(\frac{\alpha(rm+h)+\eta_2\overline{\alpha}n}{b}\right)\int_0^\infty V(y)h\left(\frac{b_0b}{C},\frac{rm+h-Yy}{b_0bC}\right)J_{\eta_2,g}\left(\frac{4\pi\sqrt{nYy}}{b}\right)dy,
\end{align*}
with $J_{\eta_2,g}$ as defined in Lemma \ref{lem.GL2Voronoi}. Repeated integration by parts gives us arbitrary savings unless \begin{align*}
    1\leq n\ll\frac{b^2X^\varepsilon}{Y}\left( 1+\frac{Y}{b_0bC}\right)^2\ll \frac{b^2X^\varepsilon}{Y}+\frac{X^{1+3\varepsilon}}{b_0^2C^2}.
\end{align*}
By the condition $C\gg X^{1/2+2\varepsilon}$, the above restriction becomes \begin{align*}
    n\ll \frac{b^2}{Y^{1-\varepsilon}} \quad \text{ and } \quad b\gg Y^{1/2-\varepsilon}.
\end{align*}

Write $r_0=(r,b)$, $b=r_0b'$, $r=r_0r'$. Applying Voronoi summation on the $m$-sum with Lemma \ref{lem.GL3Voronoi} yields
\begin{align*}
    D_{a,\mathcal{H}}(X)=&\frac{Y}{\mathcal{C}|\mathcal{H}|}\sum_{\eta_1,\eta_2=\pm1}\sum_{b_0}\sum_{r_0r'=r}\sum_{\substack{  Y^{1/2-\varepsilon}\ll b=r_0b'\ll\frac{C}{b_0}\\(b',r')=1}}\frac{b'}{b_0b^2}\sum_{h\in\mathcal{H}}a(h)   \sum_{m_0|b'}\sum_m \frac{A(m,m_0)}{m_0m}\sum_n\la(\eta_2 n)\nonumber\\
    &\times S_0(m,n,h;r_0m_0,b'/m_0)F_1\left(\frac{m_0^2mX}{b'^3},\frac{nY}{b^2},h\right)+O\left(X^{-999}\right),
\end{align*}
with \begin{align*}
    S_0(m,n,h;r_0m_0,b'/m_0)=\sumast_{\alpha\Mod{b}}e\left(\frac{\alpha h+\eta_2\overline{\alpha} n}{b}\right)S(\overline{\alpha r'},\eta_1 m;b'/m_0)
\end{align*}
and \begin{align*}
    F_1\left(\frac{m_0^2mX}{b'^3},\frac{nY}{b^2},h  \right)=&\int_0^\infty\int_0^\infty V(x)G_{\eta_1}\left(\frac{m_0^2mX}{b'^3},Yy,h;b\right)J_{\eta_2,g}\left(\frac{4\pi\sqrt{nYy}}{b}\right)dxdy,
\end{align*}
where $G_{\eta_1}\left(\frac{m_0^2mX}{b'^3},Yy,h  \right)$ is defined in Lemma \ref{lem.GL3Voronoi} with \begin{align*}
    g\left(\frac{x}{X}\right)=V\left(\frac{x}{X}\right)h\left(\frac{b_0b}{C},\frac{rx+h-Yy}{b_0bC}\right).
\end{align*}
The function $G_{\eta_1}$ gives us arbitrary savings unless \begin{align*}
    m_0^2m\ll \frac{b'^3 }{X^{1-\varepsilon}}.
\end{align*}
Rewriting $F(u,v,h;b)=u^{-2/3}F_1(u,v,h;b)$, (\ref{GL3BesselDerProp}) and $C> X^{1/2+2\varepsilon}$ gives us for $b\gg  Y^{1/2-\varepsilon}$, \begin{align}\label{F1DerProp}
    x^j\frac{d^j}{dx^j}F(x,y,z;b)\ll ( 1+x^{1/3})^jX^\varepsilon
\end{align}
for any $j\geq0$.

Combining everything and rewriting $b'$ as $m_0b'$, we have \begin{align}
    D_{a,\mathcal{H}}(X)=&\frac{X^{2/3}Y}{\mathcal{C}|\mathcal{H}|}\sum_{\eta_1,\eta_2=\pm1}\sum_{b_0\geq1}\sum_{r_0r'=r}\sum_{m_0}\sum_{\substack{  Y^{1/2-\varepsilon}\ll b=r_0m_0b'\ll\frac{C}{b_0}\\(m_0b',r')=1}}\frac{1}{b_0m_0b'b^2}\sum_{h\in\mathcal{H}}a(h)   \sum_{m\ll\frac{m_0b'^3 }{X^{1-\varepsilon}}} A(m,m_0)\left(\frac{m_0}{m}\right)^{1/3}\nonumber\\
    &\times \sum_{n\ll\frac{b^2 }{Y^{1-\varepsilon}}}\la(\eta_2 n)S_0(m,n,h;r_0m_0,b') F\left(\frac{mX}{m_0b'^3},\frac{nY}{b^2},h;b\right)+O\left(X^{-999}\right).
\end{align}

\subsection{Cauchy-Schwarz inequality and Poisson summation}\label{sect.CS1}

Notice that by Corollary \ref{GL3RPavg}, \begin{align*}
    &\mathop{\sum\sum\sum\sum}_{\substack{b_0r_0m_0b'\ll C\\r_0r'=r}}\frac{1}{b_0m_0^{4/3}b'^2}\sum_{m\ll\frac{m_0b'^3 }{X^{1-\varepsilon}}} \frac{|A(m,m_0)|^2}{m^{2/3}}\ll X^\varepsilon\sup_{M_0C_1\ll C}\frac{1}{M_0^{4/3}C_1}\sum_{m_0\asymp M_0}\sum_{m\ll\frac{M_0C_1^3 }{X^{1-\varepsilon}}}\frac{|A(m,m_0)|^2}{m^{2/3}}\ll  X^{-1/3+\varepsilon}.
\end{align*}
Now we apply Cauchy-Schwarz inequality twice to take out the $r_0,m_0,m,a,b'$-sums and then the $n$-sum. Together with Lemma \ref{lem.GL2RPavg}, we have \begin{align*}
    D_{a,\mathcal{H}}(X)^2\ll &\frac{ XY^{1+\varepsilon}}{C^2|\mathcal{H}|^2}\sum_{\eta_1,\eta_2=\pm1}\sum_{b_0\geq1}\sum_{r_0r'=r}\sum_{m_0}\sum_{\substack{  Y^{1/2-\varepsilon}\ll b=r_0m_0b'\ll\frac{C}{b_0}\\(m_0b',r')=1}}\frac{1}{b_0b^2}\sum_{n\ll\frac{b^2 }{Y^{1-\varepsilon}}}\\
    &\times \sum_{m\ll\frac{m_0b'^3 }{X^{1-\varepsilon}}}\left|\sum_{h\in\mathcal{H}}a(h)   S_0(m,n,h;r_0m_0,b') F\left(\frac{mX}{m_0b'^3},\frac{nY}{b^2},h;b\right)\right|^2.
\end{align*}

Before we proceed, we apply a lengthening trick. Let $L\geq1$ be a parameter. Take $U\in C_c^\infty([-2,2])$ be a fixed function such that $U(x)=1$ for $-1\leq x\leq 1$, we have \begin{align}
    D_{a,\mathcal{H}}(X)^2\ll &\frac{ XY^{1+\varepsilon}}{C^2|\mathcal{H}|^2}\sum_{\eta_1,\eta_2=\pm1}\sum_{b_0\geq1}\sum_{r_0r'=r}\sum_{m_0}\sum_{\substack{  Y^{1/2-\varepsilon}\ll b=r_0m_0b'\ll\frac{C}{b_0}\\(m_0b',r')=1}}\frac{1}{b_0b^2}\sum_{n\ll\frac{b^2 }{Y^{1-\varepsilon}}}\nonumber\\
    &\times \sum_mU\left(\frac{mX^{1-\varepsilon}}{m_0b'^3 L}\right)\left|\sum_{h\in\mathcal{H}}a(h)   S_0(m,n,h;r_0m_0,b') F\left(\frac{mX}{m_0b'^3},\frac{nY}{b^2},h;b\right)\right|^2.
\end{align}

\begin{remark}
    By adding more summation terms outside an absolute value square, this lengthening trick allow us to boost the diagonal contribution while lowering the offdiagonal contribution.
\end{remark}

Opening the square and applying Poisson summation on the $m$-sum, the $m$-sum is equal to \begin{align*}
    &\sum_mU\left(\frac{mX^{1-\varepsilon}}{m_0b'^3 L}\right)S_0(m_0,m,n,h_1;b')\overline{S_0(m_0,m,n,h_2;b')}F\left(\frac{mX}{m_0b'^3},\frac{nY}{b^2},h_1;b\right)\overline{F\left(\frac{mX}{m_0b'^3},\frac{nY}{b^2},h_2;b\right)}\\
    =&\frac{m_0b'^3 L}{X^{1-\varepsilon}}\sum_m T(m,n,h_1,h_2;r_0m_0,b')G\left(m,n,h_1,h_2;b\right),
\end{align*}
where \begin{align*}
    T(m,n,h_1,h_2;r_0m_0,b')=\frac{1}{b'}\sum_{\gamma\Mod{b'}}S_0(\gamma,n,h_1;r_0m_0,b')\overline{S_0(\gamma,n,h_2;r_0m_0,b')}e\left(\frac{m\gamma}{b'}\right)
\end{align*}
and \begin{align*}
    G\left(m,n,h_1,h_2;b\right)=\int_\R U(w)F\left( LX^\varepsilon w,\frac{nY}{b^2},h_1;b\right)\overline{F\left( LX^\varepsilon w,\frac{nY}{b^2},h_2;b\right)}e\left(-\frac{m_0b'^2 Lmw}{X^{1-\varepsilon}}\right)dw.
\end{align*}
With (\ref{F1DerProp}), repeated integration by parts gives us arbitrary savings unless \begin{align*}
    |m|\ll \frac{X}{m_0b'^2 L} L^{1/3}X^\varepsilon=\frac{X^{1+\varepsilon}}{m_0b'^2 L^{2/3}}\ll \frac{m_0X^\varepsilon}{L^{2/3}}.
\end{align*}
Now we exploit our lengthening trick and take $L=m_0^{3/2}X^{3\varepsilon}$. This choice gives us arbitrary savings unless $m=0$.

Inserting the above analysis back into the bound for $D_{a,\mathcal{H}}(X)$, we have \begin{align*}
    D_{a,\mathcal{H}}(X)^2\ll &\frac{  Y^{1+\varepsilon}}{C^2|\mathcal{H}|^2}\sum_{\eta_1,\eta_2=\pm1}\sum_{b_0\geq1}\sum_{r_0r'=r}\sum_{m_0}\sum_{\substack{b=r_0m_0b'\ll\frac{C}{b_0}\\(m_0b',r')=1}}\frac{\sqrt{m_0}b'}{b_0r_0^2}\sum_{n\ll\frac{b^2 }{Y^{1-\varepsilon}}}\\
    &\times \mathop{\sum\sum}_{h_1,h_2\in\mathcal{H}}a(h_1)\overline{a(h_2)}T(0,n,h_1,h_2;r_0m_0,b')G\left(0,n,h_1,h_2;b\right).
\end{align*}

\subsection{Final bound}

Applying Lemma \ref{lem.Charsum} with $\ell=0$ on $T$ and the bound $G\ll X^\varepsilon$ on the bound for $D_{a,\mathcal{H}}(X)^2$ above, we have \begin{align*}
    D_{a,\mathcal{H}}(X)^2\ll &\frac{  Y^{1+\varepsilon}}{C^2|\mathcal{H}|^2}\sum_{b_0\geq1}\sum_{r_0r'=r}\sum_{m_0}\sum_{b=r_0m_0b'\ll\frac{C}{b_0}}\frac{m_0^{3/2}b'^2}{b_0r_0}\left(\frac{b^2 }{Y}\right)^3\sum_{b''|b'}b''\\
    &\times \mathop{\sum\sum}_{h_1,h_2\in\mathcal{H}}|a(h_1)a(h_2)|\delta(h_1\equiv h_2\Mod{b''})\\
    \ll& \frac{ C^7}{|\mathcal{H}|^2X^{2-\varepsilon}}\sum_{b\ll C}\mathop{\sum\sum}_{h_1,h_2\in\mathcal{H}}|a(h_1)a(h_2)|\delta(h_1\equiv h_2\Mod{b})\\
    \ll& \frac{ C^7}{|\mathcal{H}|^2X^{2-\varepsilon}}\sum_{b\ll C}\mathop{\sum\sum}_{h_1,h_2\in\mathcal{H}}\left(|a(h_1)|^2+|a(h_2)|^2\right)\delta(h_1\equiv h_2\Mod{b})\\
    \ll&\frac{ C^7}{|\mathcal{H}|^2X^{2-\varepsilon}}\|a\|_2^2\sum_{b\ll C}\sup_{h\in\mathcal{H}}\left|\{h'\in\mathcal{H}: h'\equiv h\Mod{b})\right|.
\end{align*}
Taking $C=X^{1/2+2\varepsilon}$ yields Theorem \ref{mainthm}.

\section{Proof of Theorem \ref{mainthm2}}

With $a(h)=a_\pm(h)$ and $\mathcal{H}$ satisfying $(*)$, i.e. let $0\leq \ell\ll X^{1+\varepsilon}$ and \begin{align}
    a(h)=a_\pm(h)=\mathop{\sum\sum\sum}_{\substack{q_1\in\mathcal{Q}_1\\\pm dq_1q_2=h-\ell}}V_1\left(\frac{d}{D}\right)V_2\left(\frac{q_2}{Q_2}\right)a'(q_1)\tag{*}
\end{align}
for some sequence $\{a'(q_1)\}_{q_1\in\mathcal{Q}_1}$ with $\mathcal{Q}_1\subset\llbracket 1,Q_1\rrbracket$ and \begin{align*}
    \mathcal{H}\supset \{h\in\N: a(h)\neq0\},
\end{align*}
we can rewrite $D_{a_\pm,\mathcal{H}}(X)$ as \begin{align*}
    D_{a_\pm,\mathcal{H}}(X)=\frac{1}{|\mathcal{H}|}\sum_d V_1\left(\frac{d}{D}\right)\sum_{q_1\in\mathcal{Q}_1}a'(q_1)\sum_{q_2}V_2\left(\frac{q_2}{Q_2}\right)\sum_{m=1}^\infty A(1,m)\la(rm+\ell\pm dq_1q_2)V\left(\frac{m}{X}\right).
\end{align*}

\subsection{Refinement of Lemma \ref{mainlemma}}

We will prove Theorem \ref{mainthm2} in this section by analysing $D_{a_\pm,\mathcal{H}}(X)$ in a similar fashion as the proof of Theorem \ref{mainthm}. Before that, we first make use of the special structure of $a(h)$ and $\mathcal{H}$ to prove the following lemma.

\begin{lemma}\label{mainlemma2}
    Let $A,\varepsilon>0$, $n$ be an integer such that $|n|\ll Y\rightarrow\infty$. Let $t=Y^\varepsilon$ and $C=Y^{1/2+5\varepsilon}$. Then there exists $\mathcal{C}\asymp C$ and a fixed function $h$ such that
    \begin{align*}
        \mathcal{S}^\pm(n,q_1'):=&\sum_d V_1\left(\frac{d}{D}\right)\sum_{q_2}V_2\left(\frac{q_2}{Q_2}\right)\delta(n\pm dq_1q_2=0)=\mathcal{S}_0^\pm(n,q_1')+\mathcal{S}_1^\pm(n,q_1')+O_A\left(Y^{-A}\right),
    \end{align*}
    where \begin{align*}
        \mathcal{S}_0^\pm(n,q_1')=\frac{1}{\mathcal{C}}\sum_{c_0\geq1}\sum_{1\leq c\leq Y^{1/2-2\varepsilon}}\frac{1}{c_0c}\sum_d V_1\left(\frac{d}{D}\right)\sum_{q_2}V_2\left(\frac{q_2}{Q_2}\right)\sumast_{\alpha\Mod{c}}e\left(\frac{\alpha(n\pm dq_1q_2)}{c}\right)h\left(\frac{c_0c}{C},\frac{n\pm dq_1q_2}{c_0cC}\right)\frac{(n^2)^{it}}{(dq_1q_2)^{2it}},
    \end{align*}
    and \begin{align*}
        \mathcal{S}_1^\pm(n,q_1')=0
    \end{align*}
    unless $D+Q_2\ll X^{1/2+7\varepsilon}$, and in such a case, if $\mathcal{Q}_1\subset\{1\text{ or Primes in }[Q_1,2Q_1]\}$ and $Q_1Q_2\gg X^{1/2+\varepsilon}$, we have \begin{align*}
        \mathcal{S}_1^\pm(n,q_1')=&\frac{DQ_2}{\mathcal{C}}\sum_{c_0\geq1}\sum_{Y^{1/2-2\varepsilon}< c\ll \frac{C}{c_0}}\frac{1}{c_0c^2}\sum_{0<d\asymp\frac{cY^\varepsilon}{D}}\sum_{0<q_2\asymp\frac{cY^\varepsilon}{Q_2}}S(n,\mp dq_2\overline{q_1};c)\\
        &\times \int_0^\infty\int_0^\infty V_1(x)V_2(y)h\left(\frac{c_0c}{C},\frac{n\pm Dq_1Q_2xy}{c_0cC}\right)e\left(\frac{dDx+q_2Q_2y}{c}\right)\left(\frac{n^2}{(Dq_1Q_2xy)^2)}\right)^{it}dxdy.
    \end{align*}
\end{lemma}

For a better presentation, we postpone the proof of Lemma \ref{mainlemma2} to the end of this section (Section \ref{sect.Proofmainlemma2}).

Writing $Y=rX+\ell$, then $X\ll Y\ll X^{1+\varepsilon}$. Write $t=X^\varepsilon$ and $C=Y^{1/2+5\varepsilon}$. Take a fixed $\varphi\in C_c^\infty([1/2,5/2])$ and $\varphi(x)=1$ for $1\leq x\leq 2$ as before, we can freely insert the factor $1=\varphi((rm+\ell)/Y)$ into $D_{a_\pm,\mathcal{H}}(X)$. Applying Lemma \ref{mainlemma2}, we have \begin{align*}
    D_{a_\pm,\mathcal{H}}(X)=\mathcal{D}_0(X)+\mathcal{D}_1(X)+O_A\left(X^{-A}\right)
\end{align*}
for any $A>0$, where \begin{align*}
    \mathcal{D}_j(X)=\frac{1}{|\mathcal{H}|}\sum_{q_1\in\mathcal{Q}_1}a'(q_1)\sum_m A(1,m)V\left(\frac{m}{X}\right)\sum_n\la(n)\varphi\left(\frac{n}{Y}\right)\mathcal{S}_j^\pm(rm+\ell-n,q_1')
\end{align*}
for $j=0,1$.

\subsection{Treatment of \texorpdfstring{$\mathcal{D}_0(X)$}{D0(X)}}

By the definition of $\mathcal{S}_0^\pm(rm+\ell-n,q_1')$ given in Lemma \ref{mainlemma2}, we see that the $n$-sum in $\mathcal{D}_0(X)$ is given by \begin{align*}
    \sum_n\la(n)\varphi\left(\frac{n}{Y}\right)e\left(\frac{\alpha n}{c}\right)h\left(\frac{c_0c}{C},\frac{rm+\ell-n\mp dq_1q_2}{c_0cC}\right)((rm+\ell-n)^2)^{it}.
\end{align*}
Applying Voronoi summation (Lemma \ref{lem.GL2Voronoi}), this is equal to \begin{align*}
    \frac{Y}{c}\sum_{\eta=\pm1}\sum_n\la(\eta n)e\left(-\eta\frac{\overline{\alpha} n}{c}\right)\int_0^\infty\varphi(y)h\left(\frac{c_0c}{C},\frac{rm+\ell-Yy\mp dq_1q_2}{c_0cC}\right)((rm+\ell-Yy)^2)^{it}J_{\eta,g}\left(\frac{4\pi\sqrt{Yny}}{c}\right)dy.
\end{align*}
Now with the restriction $c\ll Y^{1/2-2\varepsilon}$ and $C\gg X^{1/2+3\varepsilon}$, repeated integration by parts gives us arbitrary savings unless \begin{align*}
    n\ll\frac{c^2}{Y^{1-\varepsilon}}\left(X^\varepsilon+\frac{Y}{c_0cC}\right)^2\ll \frac{c^2}{Y^{1-\varepsilon}}+\frac{X^{1+3\varepsilon}}{c_0^2C^2}\ll X^{-\varepsilon}.
\end{align*}
Hence we have \begin{align*}
    \mathcal{D}_0(X)\ll_A X^{-A}
\end{align*}
and \begin{align}\label{temp}
    D_{a_\pm,\mathcal{H}}(X)=\mathcal{D}_1(X)+O_A\left(X^{-A}\right)
\end{align}
for any $A>0$.

\subsection{Treatment of \texorpdfstring{$\mathcal{D}_1(X)$}{D1(X)}}

At this point, we already obtain \begin{align*}
    D_{a_\pm,\mathcal{H}}(X)\ll_A X^{-A}
\end{align*}
when $D+Q_2\gg X^{1/2+\varepsilon}$ as $\mathcal{S}_1^\pm(rm+\ell-n,q_1')=0$ in such a case\footnote{$D+Q_2\gg X^{1/2+7\varepsilon}$ to be exact but one can easily adjust the $\varepsilon$ or just abuse the $\varepsilon$-convention.}. We are left to prove the second statement, which we now further assume $\mathcal{Q}_1\subset\{1\text{ or Primes in }[Q_1,2Q_1]\}$ and $Q_1Q_2\gg X^{1/2+\varepsilon}$.

Write \begin{align*}
    F_0(rm,n,\ell;c)=\int_0^\infty\int_0^\infty V_1(x)V_2(y)h\left(\frac{c_0c}{C},\frac{rm+\ell-n\pm Dq_1Q_2xy}{c_0cC}\right)e\left(\frac{dDx+q_2Q_2y}{c}\right)\left(\frac{(rm+\ell-n)^2}{(Dq_1Q_2xy)^2)}\right)^{it}dxdy,
\end{align*}
then \begin{align*}
    \mathcal{D}_1(X)=&\frac{DQ_2}{\mathcal{C}|\mathcal{H}|}\sum_{c_0\geq1}\sum_{Y^{1/2-2\varepsilon}<c\ll \frac{C}{c_0}}\frac{1}{c_0c^2}\sum_{q_1\in\mathcal{Q}_1}a'(q_1)\sum_{0<d\asymp\frac{cY^\varepsilon}{D}}\sum_{0<q_2\asymp\frac{cY^\varepsilon}{Q_2}}\\
    &\times \sum_m A(1,m)V\left(\frac{m}{X}\right)\sum_n\la(n)\varphi\left(\frac{n}{Y}\right)S(rm+\ell-n,\mp dq_2\overline{q_1};c)F_0(rm,n,\ell;c).
\end{align*}
By the restriction of $c> Y^{1/2-2\varepsilon}$, $F_0$ is $X^{3\varepsilon}$-inert.

From this point onward, we proceed in the exact same way as the proof of Theorem \ref{mainthm} in Section \ref{sect.ProofMainThm}. As it is repetitive, we will only write down the key steps.

\subsubsection{Voronoi summations}

Applying Voronoi summations in the $n$-sum and then the $m$-sum, we get \begin{align*}
    \mathcal{D}_1(X)=&\frac{DQ_2X^{2/3}Y}{\mathcal{C}|\mathcal{H}|}\sum_{\eta_1,\eta_2=\pm1}\sum_{c_0\geq1}\sum_{r_0r'=r}\sum_{m_0}\sum_{\substack{Y^{1/2-2\varepsilon}<c=r_0m_0c'\ll \frac{C}{c_0}\\(m_0c',r')=1}}\frac{1}{c_0m_0c'c^3}\sum_{q_1\in\mathcal{Q}_1}a'(q_1)\sum_{0<d\asymp\frac{cY^\varepsilon}{D}}\sum_{0<q_2\asymp\frac{cY^\varepsilon}{Q_2}}\\
    &\times \sum_{m\ll \frac{m_0c'^3}{X^{1-\varepsilon}}} A(1,m)\sum_{n\ll\frac{b^2}{Y^{1-\varepsilon}}}\la(\eta_2 n)S_\ell(m,n,\mp dq_2\overline{q_1};r_0m_0,c')\tilde{F}\left(\frac{mX}{m_0c'^3},\frac{nY}{c^2},d,q_1,q_2;c\right)+O_A\left(X^{-A}\right),
\end{align*}
where \begin{align*}
    S_\ell(m,n,\mp dq_2\overline{q_1};r_0m_0,c')=\sumast_{\alpha\Mod{c}}e\left(\frac{\alpha\ell+\overline{\alpha}(\eta_2 n\mp dq_2\overline{q_1})}{c}\right)S(\overline{\alpha r'},\eta_1 m;c')
\end{align*}
and $\tilde{F}\left(\frac{mX}{m_0c'^3},\frac{nY}{c^2},d,q_1,q_2;c\right)$ is some function satisfying \begin{align}\label{F2DerProp}
    x^j\frac{d^j}{dx^j}\tilde{F}(x,n,d,q_1,q_2;c)\ll  ( X^\varepsilon+x^{1/3})^j X^\varepsilon
\end{align}
for any $j\geq0$.

\subsubsection{Cauchy-Schwarz inequality and Poisson summation}\label{sect.CS2}

Applying Cauchy-Schwarz inequality to take out the $m$-sum together with a lengthening parameter $L\geq1$, we have \begin{align*}
    \mathcal{D}_1(X)^2\ll &\frac{ DQ_2XY^{1+\varepsilon}}{C^2|\mathcal{H}|^2}\sum_{\eta_1,\eta_2=\pm1}\sum_{c_0\geq1}\sum_{r_0r'=r}\sum_{m_0}\sum_{\substack{  c=r_0m_0c'\ll\frac{C}{c_0}\\(m_0c',r')=1}}\frac{1}{c_0c^4}\sum_{n\ll\frac{c^2 }{Y^{1-\varepsilon}}}\\
    &\times \sum_mU\left(\frac{mX^{1-\varepsilon}}{m_0c'^3 L}\right)\left|\sum_{q_1\in\mathcal{Q}_1}a'(q_1)\sum_{0<d\asymp\frac{cY^\varepsilon}{D}}\sum_{0<q_2\asymp\frac{cY^\varepsilon}{Q_2}} S_\ell(m,n,\mp dq_2\overline{q_1};r_0m_0,c') \tilde{F}\left(\frac{mX}{m_0c'^3},\frac{nY}{c^2},d,q_1,q_2;c\right)\right|^2.
\end{align*}

Opening the square, applying Poisson summation on the $m$-sum, together with $L=m_0^{2/3}X^{5\varepsilon}$ giving arbitrary savings except the zero frequency, we get \begin{align*}
    \mathcal{D}_1(X)^2\ll &\frac{ D^2Q_2^2Y^{1+\varepsilon}}{C^2|\mathcal{H}|^2}\|a\|_\infty^2\sum_{\eta_1,\eta_2=\pm1}\sum_{c_0\geq1}\sum_{r_0r'=r}\sum_{m_0}\sum_{\substack{c=r_0m_0c'\ll\frac{C}{c_0}\\(m_0b',r')=1}}\frac{1}{c_0r_0^4m_0^{3/2}c'}\sum_{n\ll\frac{c^2 }{Y^{1-\varepsilon}}}\\
    &\times \mathop{\sum\sum}_{q_1,q_1'\in\mathcal{Q}_1}\mathop{\sum\sum}_{0<d_1,d_2\asymp\frac{cY^\varepsilon}{D}}\mathop{\sum\sum}_{0<q_2,q_2'\asymp\frac{cY^\varepsilon}{Q_2}}\left|T_2(\ell,n,d_1,d_2,q_1,q_1',q_2,q_2';c')\right|,
\end{align*}
with \begin{align*}
    T_2(\ell,n,\mp d_1q_2\overline{q_1},\mp d_2q_2'\overline{q_1'};r_0m_0,c')=\frac{1}{c'}\sum_{\gamma\Mod{c'}}S_\ell(\gamma,n,\mp d_1q_2\overline{q_1};r_0m_0,c')\overline{S_\ell(\gamma,n,\mp d_2q_2'\overline{q_1'};r_0m_0,c')}.
\end{align*}

\subsubsection{Final bound}

Applying Lemma \ref{lem.Charsum} on $T_2$, we have \begin{align*}
    \mathcal{D}_1(X)^2\ll &\frac{  D^2Q_2^2Y^{1+\varepsilon}}{C^2|\mathcal{H}|^2}\|a\|_\infty^2\sum_{c_0\geq1}\sum_{r_0r'=r}\sum_{m_0}\sum_{c=r_0m_0c'\ll\frac{C}{c_0}}\frac{1}{c_0r_0^3\sqrt{m_0}}\left(\frac{c^2 }{Y}\right)^3\sum_{c''|c'}c''\\
    &\times \mathop{\sum\sum}_{q_1,q_1'\in\mathcal{Q}_1}\mathop{\sum\sum}_{0<d_1,d_2\asymp\frac{cY^\varepsilon}{D}}\mathop{\sum\sum}_{0<q_2,q_2'\asymp\frac{cY^\varepsilon}{Q_2}}\delta(d_1q_1'q_2\equiv d_2q_1q_2'\Mod{c''})\\
    \ll& \frac{ C^5D^2Q_2^2}{|\mathcal{H}|^2X^{2-\varepsilon}}\|a\|_\infty^2\sum_{c\ll C}\mathop{\sum\sum}_{q_1,q_1'\in\mathcal{Q}_1}\mathop{\sum\sum}_{0<d_1,d_2\asymp\frac{CN^\varepsilon}{D}}\mathop{\sum\sum}_{0<q_2,q_2'\asymp\frac{CN^\varepsilon}{Q_2}}\delta(d_1q_1'q_2\equiv d_2q_1q_2'\Mod{c})\\
    \ll& \frac{ C^5D^2Q_2^2}{|\mathcal{H}|^2X^{2-\varepsilon}}\|a\|_\infty^2\sum_{c\ll C} Q_1\frac{C}{D}\frac{C}{Q_2}\left(1+\frac{C^2Q_1}{cDQ_2}\right)\ll \frac{C^8}{|\mathcal{H}|X^{2-\varepsilon}}\left(1+\frac{CQ_1}{DQ_2}\right)\|a\|_\infty^2.
\end{align*}
Here we used $|\mathcal{H}|\gg DQ_1Q_2X^{-\varepsilon}$.

Together with (\ref{temp}), this concludes Theorem \ref{mainthm2}. Now it remains to prove Lemma \ref{mainlemma2}.

\subsection{Proof of Lemma \ref{mainlemma2}}\label{sect.Proofmainlemma2}

Write $t=Y^\varepsilon$. We start by multiplying expression by $$1=\left(\left(\frac{n}{dq_1q_2}\right)^2\right)^{it}=\frac{(n^2)^{it}}{(dq_1q_2)^{2it}},$$
    giving us \begin{align*}
    \mathcal{S}^\pm(n,q_1')=&\sum_d V_1\left(\frac{d}{D}\right)\sum_{q_2}V_2\left(\frac{q_2}{Q_2}\right)\delta(n\pm dq_1q_2=0)\frac{(n^2)^{it}}{(dq_1q_2)^{2it}}.
\end{align*}

\begin{remark}
This is a small trick to eliminate some zero frequencies that come from Poisson summations from $d$ and $q_2$ sums in the upcoming steps, which shortens our proof.
\end{remark}

Applying Lemma \ref{DeltaCor} with $C= Y^{1/2+5\varepsilon}$ and $q=1$, taking out the g.c.d. $(\alpha,c)$ in the character sum, there exists some $\mathcal{C}\asymp C$ and fixed function $h$ such that \begin{align*}
    \mathcal{S}^\pm(n,q_1')=\frac{1}{\mathcal{C}}\mathop{\sum\sum}_{c_0,c\geq1}\frac{1}{c_0c}\sum_d V_1\left(\frac{d}{D}\right)\sum_{q_2}V_2\left(\frac{q_2}{Q_2}\right)\sumast_{\alpha\Mod{c}}e\left(\frac{\alpha(n\pm dq_1q_2)}{c}\right)h\left(\frac{c_0c}{C},\frac{n\pm dq_1q_2}{c_0cC}\right)\frac{(n^2)^{it}}{(dq_1q_2)^{2it}}.
\end{align*}
Now we split the contribution depending on the size of $c$. In particular, write $\mathcal{S}_0^\pm(n,q_1')$ to be the contribution when $c\leq Y^{1/2-2\varepsilon}$, then $\mathcal{S}_0^\pm(n,q_1')$ is precisely what we need in the statement of the lemma, and \begin{align*}
    \mathcal{S}^\pm(n,q_1')=\mathcal{S}_0^\pm(n,q_1')+\mathcal{S}_1^\pm(n,q_1'),
\end{align*}
where \begin{align*}
    \mathcal{S}_1^\pm(n,q_1')=\frac{1}{\mathcal{C}}\sum_{c_0\geq1}\sum_{Y^{1/2-2\varepsilon}< c\ll \frac{C}{c_0}}\frac{1}{c_0c}\sum_d V_1\left(\frac{d}{D}\right)\sum_{q_2}V_2\left(\frac{q_2}{Q_2}\right)\sumast_{\alpha\Mod{c}}e\left(\frac{\alpha(n\pm dq_1q_2)}{c}\right)h\left(\frac{c_0c}{C},\frac{n\pm dq_1q_2}{c_0cC}\right)\frac{(n^2)^{it}}{(dq_1q_2)^{2it}}.
\end{align*}

Now we continue the analysis on $\mathcal{S}_1^\pm(n,q_1')$. Performing Poisson summation on the $q_2$-sum, we have \begin{align*}
    &\sum_{q_2} V_2\left(\frac{q_2}{Q_2}\right)e\left(\pm\frac{\alpha dq_1q_2}{c}\right)h\left(\frac{c_0c}{C},\frac{n\pm dq_1q_2}{c_0cC}\right)q_2^{-2it}\\
    =& Q_2\sum_{q_2}\delta(\alpha dq_1\equiv \pm q_2\Mod{c}) \int_0^\infty V_2(y)h\left(\frac{c_0c}{C},\frac{n\pm dq_1Q_2y}{c_0cC}\right)e\left(-\frac{t}{\pi}\log (Q_2y)+\frac{q_2Q_2y}{c}\right)dy.
\end{align*}
Repeated integration by parts (Lemma \ref{lem.statphase}) gives us arbitrary savings unless there exists $y_0\in [1,2]$ such that \begin{align*}
    \left|\frac{qQ_2}{c}-\frac{t}{\pi y_0}\right|\ll \left(1+\frac{Dq_1Q_2}{c_0cC}\right)Y^{\varepsilon/2}\ll Y^{\varepsilon/2}.
\end{align*}
Here in the last inequality we used \begin{align*}
    c_0cC\gg CY^{1/2-2\varepsilon}\gg Y^{1+3\varepsilon}.
\end{align*}
by the restriction of $c$. Together with $t=Y^{\varepsilon}$, we get arbitrary savings unless \begin{align*}
    0< q_2\asymp \frac{cY^{\varepsilon}}{Q_2}.
\end{align*}

Similarly, we perform Poisson summation on the $d$-sum to get \begin{align*}
    &\sum_d V_1\left(\frac{d}{D}\right)\delta(\alpha dq_1\equiv \pm q_2\Mod{c})h\left(\frac{c_0c}{C},\frac{n\pm dq_1Q_2y}{c_0cC}\right)d^{-2it}\\
    =&\frac{D}{c}\sum_d\sum_{\gamma\Mod{c}}\delta(\alpha \gamma q_1 \equiv \pm q_2\Mod{c})e\left(-\frac{d\gamma}{c}\right)\int_0^\infty V_1(x)h\left(\frac{c_0c}{C},\frac{n\pm Dq_1Q_2xy}{c_0cC}\right)e\left(-\frac{t}{\pi}\log (Dx)+\frac{dDx}{c}\right)dx.
\end{align*}
Applying the same integral analysis gives us arbitrary savings unless \begin{align*}
    0<d\asymp\frac{cY^{\varepsilon}}{D}.
\end{align*}

Combining everything, we have \begin{align*}
    \mathcal{S}_1^\pm(n,q_1')=&\frac{DQ_2}{\mathcal{C}}\sum_{c_0\geq1}\sum_{Y^{1/2-2\varepsilon}< c\ll \frac{C}{c_0}}\frac{1}{c_0c^2}\sum_{0<d\asymp\frac{cY^\varepsilon}{D}}\sum_{0<q_2\asymp\frac{cY^\varepsilon}{Q_2}}\mathcal{C}(n,d,q_1,q_2;c)\\
    &\times \int_0^\infty\int_0^\infty V_1(x)V_2(y)h\left(\frac{c_0c}{C},\frac{n\pm Dq_1Q_2xy}{c_0cC}\right)e\left(\frac{dDx+q_2Q_2y}{c}\right)\left(\frac{n^2}{(Dq_1Q_2xy)^2)}\right)^{it}dxdy+O_A\left(Y^{-A}\right)
\end{align*}
for any $A>0$, where \begin{align*}
    \mathcal{C}(n,d,q_1,q_2;c)=\sumast_{\alpha\Mod{c}}\sum_{\gamma\Mod{c}}\delta(\alpha\gamma q_1\equiv\pm q_2\Mod{c})e\left(\frac{\alpha n-d\gamma}{c}\right).
\end{align*}

If $D+Q_2>X^{1/2+7\varepsilon}$, we immediately get \begin{align*}
    \mathcal{S}_1^\pm(n,q_1')\ll_A Y^{-A}
\end{align*}
for any $A>0$ by the restrictions of the $d$ and $q_2$ sums.

Finally, we further suppose $\mathcal{Q}_1\subset \{1\text{ or Primes in }[Q_1,2Q_1]\}$ and simplify the character sum. First notice if $q_1|c$, the congruence condition implies $q_1|q_2$. However, we have $0<q_2\ll \frac{cY^\varepsilon}{Q_2}\ll\frac{CY^\varepsilon}{Q_2}\ll Q_1Y^{-\varepsilon}$, which is not possible. Hence we have $(c,q_1)=1$. Now the congruence condition implies \begin{align*}
    \gamma\equiv \pm\overline{\alpha q_1} q_2\Mod{c}.
\end{align*}
This gives us \begin{align*}
    \mathcal{C}(n,d,q_1,q_2;c)=S(n,\mp dq_2\overline{q_1};c),
\end{align*}
and thus concludes the proof.

\section{Recovering Munshi's fixed shift bound without Jutila's circle method}\label{sect.RecoveringMunshi}

In this section, we sketch the simple steps needed to get to Munshi's bound for fixed shift, i.e. \begin{align*}
    D_\ell(X):=\sum_{m=1}^\infty A(1,m)\la(m+\ell)V\left(\frac{m}{X}\right)\ll X^{1-1/26+\varepsilon},
\end{align*}
by using his analysis in \cite{munshi2013shifted} together with Lemma \ref{mainlemma2} and Theorem \ref{mainthm2}, but without using Jutila's circle method.

Take a fixed $U\in C_c^\infty(\R)$ even such that $U(x)=1$ for $-2\leq x\leq 2$ as before. Let $Q_1,Q_2>0$ be parameters such that $Q_1^2\ll X^{1/2+\varepsilon}\ll Q_1Q_2$, $D=X^{1-\varepsilon}/Q_1Q_2$. Let $\mathcal{Q}_1=\{\text{Primes in }[Q_1,2Q_1]\}$. Applying Lemma \ref{mainlemma} with $F=U$, $\mathcal{Q}=\mathcal{Q}_1\cdot \llbracket Q_2,2Q_2\rrbracket$, \begin{align*}
    b(q)=\sum_{\substack{q_1q_2=q\\q_1\in\mathcal{Q}_1}}V_2\left(\frac{q_2}{Q_2}\right)
\end{align*}
with some fixed $V_2\not\equiv0\in C_c^\infty([1,2])$, we obtain \begin{align*}
    D_\ell(X)=M.T.-A.S.^+-A.S.^-,
\end{align*}
where \begin{align*}
    M.T.:=\frac{1}{\mathfrak{Q}}\sum_{q_1\in\mathcal{Q}_1}\sum_{q_2}V_2\left(\frac{q_2}{Q_2}\right)\sum_m A(1,m)V\left(\frac{m}{X}\right)\sum_n\la(n)\varphi \left(\frac{n}{Y}\right)\delta(m+\ell\equiv n\Mod{q_1q_2})U\left(\frac{m+\ell-n}{Dq}\right)
\end{align*}
and \begin{align*}
    A.S.^\pm:=\frac{1}{\mathfrak{Q}}\sum_{q_1\in\mathcal{Q}_1}\sum_{q_2}V_2\left(\frac{q_2}{Q_2}\right)\sum_d U\left(\frac{d}{D}\right)\sum_m A(1,m)V\left(\frac{m}{X}\right)\la(m+\ell\pm dq_1q_2)\varphi \left(\frac{m+\ell\pm dq_1q_2}{Y}\right),
\end{align*}
with $\mathfrak{Q}=\sum_{q\in\mathcal{Q}}b(q)\gg |\mathcal{Q}_1|Q_2$ as defined in Lemma \ref{mainlemma}.

Applying a smooth dyadic subdivision on the $d$-sum, we have \begin{align*}
    A.S.^\pm \ll& \sup_{D'\ll D}\frac{X^\varepsilon}{|\mathcal{Q}_1|Q_2}\sum_{q_1\in\mathcal{Q}_1}\sum_{q_2}V_2\left(\frac{q_2}{Q_2}\right)\sum_d V_1\left(\frac{d}{D'}\right)\sum_m A(1,m)V\left(\frac{m}{X}\right)\la(m+\ell\pm dq_1q_2)\varphi \left(\frac{m+\ell\pm dq_1q_2}{Y}\right)\\
    \ll &\sup_{D'\ll D} X^\varepsilon D' \left|D_{a_\pm,\mathcal{H}}(X)\right|,
\end{align*}
for some appropriate $1$-inert function $V_1\in C_c^\infty(\R^+)$, with $\{a_\pm(h)\}$ and $\mathcal{H}=\{h:a_\pm(h)\neq0\}$ chosen as in $(*)$ (with D' in place of $D$). Then $|\mathcal{H}|\gg D'Q_1Q_2X^{-\varepsilon}$. Applying Theorem \ref{mainthm2}, we obtain the bound \begin{align*}
    A.S.^\pm \ll \sup_{D'\ll D} D'\frac{X^{1+\varepsilon}}{\sqrt{D'Q_1Q_2}}\left(1+\frac{\sqrt{X}Q_1}{D'Q_2}\right)^{1/2}\ll \frac{X^{3/2}}{Q_1Q_2}.
\end{align*}
Choosing $Q_1Q_2=X^{1/2+\delta}$, we obtain \begin{align*}
    D_\ell(X)=M.T.+O\left(X^{-\delta}\right).
\end{align*}
As a result, we have recovered the same bound of the error term from Jutila's circle method, and the $M.T.$ that remains essentially matches with the main term of Jutila's circle method. 
Indeed, detecting the congruence condition by additive characters \begin{align*}
    \delta(m+\ell\equiv n\Mod{q_1q_2})=\frac{1}{q_1q_2}\sum_{q|q_1q_2}\sumast_{\alpha\Mod{q}}e\left(\frac{\alpha(m+\ell-n)}{q}\right),
\end{align*}
we have \begin{align*}
    M.T.:=\frac{1}{\mathfrak{Q}}\sum_{q_1\in\mathcal{Q}_1}\sum_{q_2}\frac{1}{q_1q_2}V_2\left(\frac{q_2}{Q_2}\right)\sum_{q|q_1q_2}\sum_m A(1,m)V\left(\frac{m}{X}\right)\sum_n\la(n)\varphi \left(\frac{n}{Y}\right)\sumast_{\alpha\Mod{q}}e\left(\frac{\alpha(m+\ell-n)}{q}\right)U\left(\frac{m+\ell-n}{Dq}\right).
\end{align*}
The only difference between our $M.T.$ here and the main term in Jutila's circle method used by Munshi in \cite{munshi2013shifted} is: \begin{itemize}
    \item $q_2$ sum over all integers between $Q_2$ to $2Q_2$ instead of primes only, and
    \item the modulus q is a divisor of $q_1q_2$ instead precisely $q_1q_2$.
\end{itemize}
However, such a difference is not essential in Munshi's treatment, and the same analysis goes through with small technicalities, giving us the same bound for $M.T.$ as his $\tilde{D}_h(X)$ (with $h=\ell$ in his notation). Hence we recovers Munshi's result \begin{align*}
    D_\ell(X)\ll X^{1-1/26+\varepsilon}
\end{align*}
using an average shifted sum analysis instead of Jutila's circle method.

\appendix

\section{Reformulation of the Delta method}\label{sect.ProofDeltaMethod}

\begin{lemma}[Reformulation of the Delta method]\label{DeltaMethod}
    Let $\varepsilon>0$, $n,q$ be integers such that $r>0$ and $|n|\ll N\rightarrow\infty$. Let $C,D>0$ be parameters such that $C>N^\varepsilon$. Let $U, W\not\equiv 0$ be any smooth even functions such that $U(0)=1, W(0)=0$ and $U$ decays exponentially at $\infty$. Then we have, \begin{align*}
        \delta(n=0)=S_1-S_2,
    \end{align*}
    where \begin{align*}
        S_1=\frac{1}{\mathcal{C}}\sum_{c\geq1}\delta\left(n\equiv 0\bmod{cq}\right)W(c)U\left(\frac{n}{cDq}\right)U\left(\frac{c}{C}\right)
    \end{align*}
    and \begin{align*}
        S_2=\frac{1}{\mathcal{C}}\sum_{d\geq1}\delta(n\equiv 0\bmod{dq})W\left(\frac{n}{dq}\right)U\left(\frac{n}{dDq}\right)U\left(\frac{d}{D}\right),
    \end{align*}
    with \begin{align*}
        \mathcal{C}=\sum_{c\geq1}W(c)U\left(\frac{c}{C}\right).
    \end{align*}
\end{lemma}

Taking $C=D$, $W(x)=W'\left(\frac{x}{C}\right)$ with $W'\in C_c^\infty([-2,-1]\bigcup [1,2])$ even and detecting the congruence condition by additive characters yields Lemma \ref{DeltaCor}.

\begin{proof}
Let \begin{align*}
    \mathcal{D}=\delta(n=0)\sum_{c\geq1}W(c)U\left(\frac{c}{C}\right),
\end{align*}
\begin{align*}
    \tilde{S_2}=\sum_{0\neq d\in\Z}\sum_{c\geq1}\delta(n=cdq)W(c)U\left(\frac{n}{cDq}\right)U\left(\frac{n}{Cdq}\right)
\end{align*}
and define \begin{align*}
    \tilde{S_1}=\mathcal{D}+\tilde{S_2}.
\end{align*}

We start by rewriting $d=\frac{n}{cq}$ in $\tilde{S_2}$, we get \begin{align*}
    \tilde{S_2}=&\sum_{c\geq1}\delta\left(n\equiv 0\bmod{cq}, n\neq 0\right)W(c)U\left(\frac{n}{cDq}\right)U\left(\frac{c}{C}\right).
\end{align*}
Hence adding up $\mathcal{D}$ and $\tilde{S_2}$ with $U(0)=1$, we get \begin{align}
    \tilde{S_1}=\sum_{c\geq1}\delta\left(n\equiv 0\bmod{cq}\right)W(c)U\left(\frac{n}{cDq}\right)U\left(\frac{c}{C}\right).
\end{align}
On the other hand, if we start by rewriting $c=\frac{n}{dq}$ in $\tilde{S_2}$, we get \begin{align*}
    \tilde{S_2}=\sum_{0\neq d\in\Z}\delta\left(n\equiv 0\bmod{dq}, \frac{n}{d}>0\right)W\left(\frac{n}{dq}\right)U\left(\frac{n}{Cdq}\right)U\left(\frac{d}{D}\right).
\end{align*}
If $n>0$, we get \begin{align*}
    \tilde{S_2}=\sum_{d\geq 1}\delta\left(n\equiv 0\bmod{dq}\right)W\left(\frac{n}{dq}\right)U\left(\frac{n}{Cdq}\right)U\left(\frac{d}{D}\right).
\end{align*}
and if $n<0$, we get \begin{align*}
    \tilde{S_2}=&\sum_{d\geq1}\delta\left(n\equiv 0\bmod{-dq}\right)W\left(-\frac{n}{dq}\right)U\left(-\frac{n}{Cdq}\right)U\left(-\frac{d}{D}\right)\\
    =&\sum_{d\geq 1}\delta\left(n\equiv 0\bmod{dq}\right)W\left(\frac{n}{dq}\right)U\left(\frac{n}{Cdq}\right)U\left(\frac{d}{D}\right)
\end{align*}
as $U$ and $W$ are even. Combining both cases above with the case $n=0$ giving $0$ in $\tilde{S_2}$ as $W(0)=0$, \begin{align}
    \tilde{S_2}=\sum_{d\geq 1}\delta\left(n\equiv 0\bmod{dq}\right)W\left(\frac{n}{dq}\right)U\left(\frac{n}{Cdq}\right)U\left(\frac{d}{D}\right).
\end{align}
Dividing both sides by \begin{align*}
    \mathcal{C}:=\sum_{c\geq1}W(c)U\left(\frac{c}{C}\right),
\end{align*}
we obtain Lemma \ref{DeltaMethod}.
\end{proof}

\section{Character sum analysis}

Let $\ell$ be an integer. Consider the character sum \begin{align*}
    \mathcal{T}(\ell,n,h_1,h_2;r_0m_0,b')=&\frac{1}{b'}\sum_{\gamma\Mod{b'}}\mathop{\sumast\quad\sumast}_{\alpha_1,\alpha_2\Mod{r_0m_0b'}}e\left(\frac{(\alpha_1-\alpha_2)\ell+\alpha_1 h_1-\alpha_2 h_2+\eta_2(\overline{\alpha_1}-\overline{\alpha_2}) n}{r_0m_0b'}\right)\\
    &\times S(\overline{\alpha_1 r'},\eta_1 \gamma;b')S(\overline{\alpha_2 r'},\eta_1 \gamma;b'),
\end{align*}
we have \begin{lemma}\label{lem.Charsum}
    \begin{align*}
        \mathcal{T}(\ell,n,h_1,h_2;r_0m_0,b')\ll r_0m_0b'n^2X^\varepsilon\sum_{b''|b}b''\delta(h_1\equiv h_2\Mod{b''}).
    \end{align*}
\end{lemma}

\begin{proof}
Opening the Kloosterman sums and summing over $\gamma$ gives us \begin{align*}
    \frac{1}{b'}\sum_{\gamma\Mod{b'}}e\left(\eta_1\frac{(\overline{\beta_1}-\overline{\beta_2})\gamma)}{b'}\right)=\delta(\beta_1\equiv\beta_2\Mod{b'}).
\end{align*}
This gives us \begin{align*}
    \mathcal{T}(\ell,n,h_1,h_2;r_0m_0,b')=&\mathop{\sumast\quad\sumast}_{\alpha_1,\alpha_2\Mod{r_0m_0b'}}e\left(\frac{(\alpha_1-\alpha_2)\ell+\alpha_1 h_1-\alpha_2 h_2+\eta_2(\overline{\alpha_1}-\overline{\alpha_2}) n}{r_0m_0b'}\right)\sumast_{\beta\Mod{b'}}e\left(\frac{\beta\overline{r'}(\overline{\alpha_1}-\overline{\alpha_2})}{b'}\right)\\
    =&\sum_{b_1'b_2'=b'}\mu(b_1')b_2'\mathop{\sumast\quad\sumast}_{\substack{\alpha_1,\alpha_2\Mod{r_0m_0b'}\\\alpha_1\equiv\alpha_2\Mod{b_2'}}}e\left(\frac{(\alpha_1-\alpha_2)\ell+\alpha_1 h_1-\alpha_2 h_2+\eta_2(\overline{\alpha_1}-\overline{\alpha_2}) n}{r_0m_0b'}\right).
\end{align*}
Write $b_0=(r_0m_0b_1',b_2')$, $r_0m_0b_1'=b_0b_{1,0}b_1$, $b_2'=b_0b_{2,0}b_2$ such that $b_{1,0}b_{2,0}|b_0^\infty$ and $(b_1b_2,b_0)=1$. Then the character sum splits into \begin{align*}
    \mathcal{T}(\ell,n,h_1,h_2;r_0m_0,b')=&\mathop{\sum\sum\sum\sum\sum}_{\substack{b_0b_{1,0}b_1=r_0m_0b_1'\\b_0b_{2,0}b_1'b_2=b'\\b_{1,0}b_{2,0}|b_0^\infty\\(b_2,b_0b_1')=(b_1,b_0b_2)=1)}}\mu(b_1')b_0b_{2,0}b_2\sumast_{\alpha\Mod{b_2}}e\left(\frac{\alpha (h_1-h_2)}{b_2}\right)\\
    &\times S(h_1+\ell,\eta_2 n\overline{b_0^2b_{1,0}b_{2,0}b_2}^2;b_1)S(h_2+\ell,\eta_2 n\overline{b_0^2b_{1,0}b_{2,0}b_2}^2;b_1)\\
    &\times \mathop{\sumast\sumast}_{\substack{\beta_1,\beta_2\Mod{b_0^2b_{1,0}b_{2,0}}\\\beta_1\equiv\beta_2\Mod{b_0b_{2,0}}}}e\left(\frac{(\beta_1 (h_1+\ell)-\beta_2 (h_2+\ell)+\eta_2(\overline{\beta_1}-\overline{\beta_2}) n)\overline{b_1b_2}}{b_0^2b_{1,0}b_{2,0}}\right).
\end{align*}

If $b_1|n$, then \begin{align*}
    S(h_1+\ell,\eta_2 n\overline{b_0^2b_{1,0}b_{2,0}b_2}^2;b_1)S(h_2+\ell,\eta_2 n\overline{b_0^2b_{1,0}b_{2,0}b_2}^2;b_1)\ll n^2.
\end{align*}
For $j=1,2$, if $b_1|(h_j+\ell)$, then \begin{align*}
    S(h_j+\ell,\eta_2 n\overline{b_0^2b_{1,0}b_{2,0}b_2}^2;b_1)=\sumast_{\beta\Mod{b_1}}e\left(\frac{\beta n}{b_1}\right)\ll nX^\varepsilon.
\end{align*}
Bounding the Kloosterman sum by the Weil bound when $b_1\not|n(h_1+\ell)(h_2+\ell)$ and evaluating the Ramanujan sum, we have \begin{align*}
    \mathcal{T}(\ell,n,h_1,h_2;r_0m_0,b')\ll& \mathop{\sum\sum\sum\sum\sum\sum}_{\substack{b_0b_{1,0}b_1=r_0m_0b_1'\\b_0b_{2,0}b_1'b_2b_3=b'\\b_{1,0}b_{2,0}|b_0^\infty\\(b_2b_3,b_0b_1')=(b_1,b_0b_2)=1)}}b_0^2b_{1,0}b_{2,0}(b_1+n^2)b_2b_3^2\delta(h_1\equiv h_2\Mod{b_3})\\
    \ll& r_0m_0b'n^2X^\varepsilon\sum_{b''|b}b''\delta(h_1\equiv h_2\Mod{b''}).
\end{align*}
\end{proof}

\printbibliography

\end{document}